\documentclass{article}
\usepackage[utf8]{inputenc}
\usepackage{appendix}
\usepackage{float}
\usepackage{enumerate}
\usepackage{todonotes}
\usepackage{comment}

\title{$H^2$-regularity of Steklov eigenfunctions on convex domains via Rellich-Pohozaev identity}
\author{Pier Domenico Lamberti\footnote{Università degli Studi di Padova\,, Dipartimento di Tecnica e Gestione dei Sistemi Industriali (DTG)\,, Stradella S. Nicola 3\,, 36100 Vicenza\,, Italy. Email:\, pierdomenico.lamberti@unipd.it}   and Luigi Provenzano\footnote{Sapienza Universit\`a di Roma\,, Dipartimento di Scienze di Base e Applicate per l'Ingegneria\,, Via Scarpa 16\,, 00161 Roma\,, Italy. Email:\, luigi.provenzano@uniroma1.it}}
\date{\today}

\usepackage{makeidx}
\usepackage{amssymb , amsmath, amsthm}
\usepackage{graphicx}
\usepackage{xcolor}
\newtheorem{defi}{Definition} 
\newtheorem{thm}[defi]{Theorem}

\newtheorem{rem}[defi]{Remark}
 
\newtheorem{lemme}[defi]{Lemma}

\newcommand{\matrice}{\begin{pmatrix}}
\newcommand{\ok}{\end{pmatrix}}

\begin{document}

\maketitle

\noindent
{\bf Abstract.} We prove that the Steklov eigenfunctions on convex domains of $\mathbb R^n$ are $H^2$ regular by adapting a classical argument combined with the Rellich-Pohozaev identity.
\vspace{11pt}


\vspace{6pt}

\section{Introduction}
Let $\Omega$ be a bounded convex open subset of $\mathbb R^n$, briefly a bounded convex domain.
In this note we study the $H^2$-regularity of the solutions $u\in H^1(\Omega)$ of the following variational problem: 
\begin{equation}\label{weak_steklov}
\int_{\Omega}\nabla u\cdot\nabla\phi=\sigma\int_{\partial\Omega}u\phi\,,\ \ \ \forall\phi\in H^1(\Omega).
\end{equation}
Here the unknown is the couple $(u,\sigma)\in H^1(\Omega)\times\mathbb R$, where $u$ is the eigenfunction and $\sigma$ is the corresponding eigenvalue. In this paper  it is always understood that $n\geq 2$ and  $H^m(\Omega)$ denotes the usual Sobolev space of functions in $L^2(\Omega)$ with weak derivatives up to order $m$ in $L^2(\Omega)$.

Since $\Omega$ is bounded and convex, it is Lipschitz regular (in the sense that locally it can be expressed as a subgraph of a Lipschitz function). Hence there exists a sequence
\begin{equation}\label{spectrum}
0=\sigma_1<\sigma_2\leq\cdots\leq\sigma_j\leq\cdots\nearrow+\infty
\end{equation}
of eigenvalues, with corresponding eigenfunctions $\{u_j\}_{j=1}^{\infty}$ which can be chosen in such a way that their traces form an orthonormal basis of $L^2(\partial\Omega)$. If the domain is smooth, weak solutions $u\in H^1(\Omega)$ of \eqref{weak_steklov} 
are regular up to the boundary of $\Omega$
 and solve the   problem
\begin{equation}\label{steklov_class}
\begin{cases}
\Delta u=0\,, &  {\rm in\ }\Omega\,,\\
\partial_{\nu}u=\sigma u\,, & {\rm on\ }\partial\Omega,
\end{cases}
\end{equation}
in the classical sense, where $\nu$ denotes the unit outer normal to $\partial\Omega$ and $\partial_{\nu}u$ the normal derivative of $u$. Problem \eqref{steklov_class} is known as the {\it Steklov problem} \cite{steklov}. See the surveys \cite{GirouSurv,Pol17} for more information on this problem and its relation with the so-called Dirichlet-to-Neumann map.

\medskip

The purpose of this note is to to prove the following

\begin{thm}\label{mainthm}
Let $\Omega$ be a bounded convex domain in $\mathbb R^n$ and let $u\in H^1(\Omega)$ be a solution of \eqref{weak_steklov}. Then $u\in H^2(\Omega)$.
\end{thm}

Theorem \ref{mainthm} is known to hold for other classical eigenvalue problems, namely, for the Dirichlet, Neumann and Robin (with a positive parameter) problems. A proof of Theorem \ref{mainthm} in these cases can be found in the classical book of Grisvard \cite[Theorem 3.1.2.3]{grisvard}.  The strategy of the proof in \cite[\S 3]{grisvard} is the following. One starts from the case of bounded smooth convex domains, where the eigenfunctions are known to be in  $H^2(\Omega)$ by elliptic regularity theory.  The $L^2$-norms of the Hessians of the eigenfunctions are  then estimated by means of explicit quantities.  
Note that in the above mentioned cases these estimates are obtained by a straightforward application of the Reilly's Formula.
Then, any given convex domain is approximated by a sequence of smooth convex domains, and the estimates on the Hessians ``pass to the limit''. 

\medskip Unfortunately, the arguments used in \cite[\S 3]{grisvard} do not apply directly to the Stkelov case, which could be seen as a Robin problem with a negative parameter. 
In this note we explain how to adapt the arguments of \cite{grisvard} to the Steklov problem. In this case, the Reilly's formula has to be accompanied by a Rellich-Pohozaev identity which allows to estimate a boundary integral. In particular, this enables us to estimate the $L^2$-norm of the Hessian of a normalized eigenfunction on a bounded smooth convex domain by means of geometric quantities -the diameter and the inradius of the domain- that are stable under the approximation procedure. Another important ingredient of our proof is the spectral stability of the Steklov problem. For this purpose, we use the recent results of \cite{bucgiatre, ferr_lamb}.
\medskip

We cannot exclude that a neat proof of the $H^2$-regularity of the Steklov eigenfunctions on convex domains
is already present in the literature but we have not been able to find it. 
\medskip

This paper is organized as follows. In Section~\ref{sec:estimates} we provide the estimates for the Hessian on smooth convex domains. In Section~\ref{approximation}
we carry out the approximation procedure and prove Theorem~\ref{mainthm}. In Section~\ref{app:A} we briefly discuss the cases of Dirichlet, Neumann and Robin boundary conditions in order to compare them with the Steklov
problem. In Appendix~\ref{app:proof} we provide a self-contained proof of a spectral stability result for  the Steklov problem adapted to the specific case under consideration.

\section{$L^2$-estimates of the Hessian on smooth convex domains}\label{sec:estimates}

In this section we establish a bound on the $L^2$-norm of the Hessian $D^2u$ of an eigenfunction $u$ of \eqref{weak_steklov} on a bounded smooth convex domain $\Omega$ which depends only on $n$, the diameter and the inradius of $\Omega$. Recall that the diameter $D$ of $\Omega$  is defined as
$$
D=\sup_{x,y\in\Omega}|x-y|
$$
and the inradius of $\Omega$ is defined as
$$
\rho=\sup_{x\in\Omega}\inf_{y\in\partial\Omega}|x-y|.
$$

Our estimates are based on two fundamental identities that we now recall: the Rellich-Pohozaev identity \cite{pohozaev,rellich} and the Reilly's formula \cite{Reilly}. We refer e.g. to \cite[Lemma~3.1]{PS_steklov} for a proof of the Rellich-Pohozaev identity used in this paper.

\begin{thm}[Rellich-Poho\v{z}aev identity]\label{rellich}
Let $\Omega$ be a bounded smooth domain in $\mathbb R^n$ and let $u\in H^2(\Omega)$ be such that $\Delta u=0$ in $L^2(\Omega)$. Then
\begin{equation}\label{rellich_id}
\int_{\partial\Omega}\partial_{\nu}u\,x\cdot\nabla u-\frac{1}{2}\int_{\partial\Omega}|\nabla u|^2\,x\cdot\nu+\frac{n-2}{2}\int_{\Omega}|\nabla u|^2=0,
\end{equation}
where $x$ denotes the position vector.
\end{thm}

\begin{thm}[Reilly's formula]
Let $\Omega$ be a bounded smooth domain in $\mathbb R^n$ and let $u\in H^2(\Omega)$. Then
\begin{equation}\label{reilly_for}
\int_{\Omega}|D^2u|^2=\int_{\Omega}(\Delta u)^2-\int_{\partial\Omega}\left[(n-1)\mathcal H(\partial_{\nu}u)^2+2\Delta_{\partial\Omega}u\partial_{\nu}u+II(\nabla_{\partial\Omega}u,\nabla_{\partial\Omega}u)\right],
\end{equation}
where $\mathcal H$ is the mean curvature of the boundary, $II$ is the second fundamental form of the boundary, $\Delta_{\partial\Omega}$ and $\nabla_{\partial\Omega}$ are the boundary Laplacian and gradient, respectively, and $|D^2u|^2=\sum_{i,j=1}^n(\partial^2_{ij}u)^2$.
\end{thm}

We note that Reilly's formula can be immediately obtained by setting $v=\nabla u$ in Theorem~3.1.1.1 in \cite{grisvard} (paying attention to the sign convention  for the second fundamental form). 

We are now ready to prove  the following
\begin{lemme}\label{hessian_smooth}
Let $\Omega$ be a bounded smooth convex domain in $\mathbb R^n$ with diameter $D$ and inradius $\rho$. Let $u$ be an eigenfunction  of \eqref{weak_steklov} corresponding to the  eigenvalue $\sigma$, normalized by $\int_{\partial\Omega}u^2=1$. Then
\begin{equation}\label{hessian_smooth_formula}
\int_{\Omega}|D^2u|^2\leq 2\sigma C_{D,\rho,\sigma}
\end{equation}
where 
\begin{equation}\label{const}
C_{D,\rho,\sigma}:=\left(D\sigma+\sqrt{D^2\sigma^2+\rho\sigma\left(D\sigma+n-2\right)}\right)^2\rho^{-2}.
\end{equation}
\end{lemme}
\begin{proof}
Since $\Omega$ is smooth, elliptic regularity implies that $u\in H^2(\Omega)$, see e.g. \cite[Chapter~2]{grisvard}. As done in \cite{PS_steklov} in another context, we begin by  using the Rellich-Pohozaev identity to show that the $L^2$-norm on $\partial\Omega$ of the normal derivative and of the tangential part of the gradient of $u$ can be estimated one in terms of the other by constants depending only on $n,\sigma,D,\rho$. Let $B_{\rho}$ be a ball of radius $\rho$ contained in $\Omega$. Up to a translation, we may assume that $0\in\Omega$ and  that this ball is centered at the origin. 
We observe that $\int_{\Omega}|\nabla u|^2=\sigma\int_{\partial\Omega}u^2=\sigma$ and that $\partial_{\nu}u=\sigma u$ on $\partial\Omega$. This last identity is understood in the sense of traces, which are well-defined in $L^2(\partial\Omega)$ since $u\in H^2(\Omega)$. After simplifications, we get from \eqref{rellich_id}
$$
\int_{\partial\Omega}|\nabla_{\partial\Omega}u|^2x\cdot\nu-2\int_{\partial\Omega}\partial_{\nu}u\nabla_{\partial\Omega}u\cdot x-\int_{\partial\Omega}(\partial_{\nu}u)^2x\cdot\nu-\sigma(n-2)=0.
$$
Now, $|x|\leq D$ and $\rho\leq x\cdot\nu \leq D$. Note that the inequality $x\cdot\nu\geq\rho$ follows from the convexity of $\Omega$, see Lemma \ref{lem_inradius} below. 
 Hence we have
$$
\int_{\partial\Omega}|\nabla_{\partial\Omega}u|^2-\frac{2D}{\rho}\left(\int_{\partial\Omega}|\nabla_{\partial\Omega}u|^2\right)^{1/2}\left(\int_{\partial\Omega}(\partial_{\nu}u))^2\right)^{1/2}-\frac{D}{\rho}\int_{\partial\Omega}(\partial_{\nu}u)^2-\frac{\sigma(n-2)}{\rho}\leq 0
$$
which reads
$$
\int_{\partial\Omega}|\nabla_{\partial\Omega}u|^2-\frac{2D\sigma}{\rho}\left(\int_{\partial\Omega}|\nabla_{\partial\Omega}u|^2\right)^{1/2}-\frac{\sigma}{\rho}\left(D\sigma+n-2\right)\leq 0.
$$
This is a second order equation in $\left(\int_{\partial\Omega}|\nabla_{\partial\Omega}u|^2\right)^{1/2}$ which implies the upper bound
\begin{equation}\label{est_tang_grad}
\int_{\partial\Omega}|\nabla_{\partial\Omega} u|^2\leq C_{D,\rho,\sigma}.
\end{equation}

Now we use this estimate to bound the $L^2$-norm of the Hessian of $u$. To do so we use Reilly's formula \eqref{reilly_for}, which in this case gives
\begin{equation}\label{reilly_appl}
\int_{\Omega}|D^2u|^2=-\int_{\partial\Omega}[\sigma^2(n-1)\mathcal H u^2+2\sigma\Delta_{\partial\Omega}uu+II(\nabla_{\partial\Omega}u,\nabla_{\partial\Omega}u)]\leq-2\sigma\int_{\partial\Omega}\Delta_{\partial\Omega}uu=2\sigma\int_{\partial\Omega}|\nabla_{\partial\Omega}u|^2
\end{equation}
Here we have used the convexity and the smoothness of the domain, namely $\mathcal H\geq 0$ on $\partial\Omega$ and $II$ is a non-negative quadratic form on the tangent spaces to $\partial\Omega$ (its eigenvalues are the principal curvatures at each point, which are all non-negative).
The conclusion follows from \eqref{est_tang_grad} and \eqref{reilly_appl}.
\end{proof}

The following lemma was used in the previous proof. 

\begin{lemme}\label{lem_inradius}
Let $\Omega$ be a bounded smooth convex set in $\mathbb R^n$ with inradius $\rho$. Let $B_{\rho}$ be a ball of radius $\rho$ and center $x_0$ contained in $\Omega$. Then
$$
(x-x_0)\cdot\nu\geq\rho
$$
on $\partial\Omega$.
\end{lemme}
\begin{proof}
Up to translation we can always assume that $x_0=0$. Let $x\in\partial\Omega$ and $\nu(x)$ be the unit outer normal to $\partial\Omega$ at $x$.  Let $T$ be the supporting hyperplane of $\Omega$ at $x$, namely $T=\{y\in\mathbb R^n:\langle \nu(x),y-x\rangle=0\}$. Consider the distance of $0$ from $T$:
$$
d(0,T)=\frac{|\langle\nu(x),-x\rangle|}{|\nu(x)|}=|\langle\nu(x),x\rangle|.
$$
Let $T_+,T_-$ be the two open half-spaces defined by $T$. More precisely,
$$
T_+=\{y\in\mathbb R^n:\langle\nu(x),y-x\rangle>0\}\,,\ \ \ T_-=\{y\in\mathbb R^n:\langle\nu(x),y-x\rangle<0\}.
$$
Clearly $\Omega\subset T_-$, hence $B_{\rho}\subset\Omega\subset T_-$ and then 
$$
|\langle\nu(x),x\rangle|=d(0,T)\geq \rho.
$$
Also, $0\in T_-$, hence 
$$
\langle\nu(x),x\rangle>0.
$$
Thus, $\langle\nu(x),x\rangle=|\langle\nu(x),x\rangle|\geq \rho$. 

\end{proof}

\section{Approximation of a convex domain by smooth convex domains and  proof of Theorem~\ref{mainthm}\label{approximation}}

In this section we prove Theorem \ref{mainthm}. The argument here is essentially the one in Grisvard's book \cite[\S3]{grisvard}. We approximate a convex domain by smooth convex domains and we pass to the limit in the estimate of the $L^2$-norm of the Hessian given by Lemma~\ref{hessian_smooth}. We recall the notion of Hausdorff distance (also known as Hausdorff-Pompeiu distance) for convex domains. 

\begin{defi}\label{HPdist}
Let $\Omega_1,\Omega_2$ be two bounded convex domains in $\mathbb R^n$. The Hausdorff distance between $\Omega_1$ and $\Omega_2$ is defined as
$$
d^{\mathcal{H}}(\Omega_1,\Omega_2):=\max\{\sup_{x\in\Omega_1}d(x,\Omega_2),\sup_{y\in\Omega_2}d(y,\Omega_1)\}.
$$
Here, for $x\in\mathbb R^n$ and $A\subset\mathbb R^n$, $d(x,A):=\inf_{a\in A}|x-a|$.
\end{defi}
Note that usually the definition of the Hausdorff distance is given for closed sets, while for open sets is given in terms of their complements. However, for open convex sets, one can use directly Definition~\ref{HPdist}.

\medskip

We recall also the following approximation result from \cite[Lemma 3.2.1.1]{grisvard}.
\begin{lemme}\label{approximation_lem}
Let $\Omega$ be a bounded convex domain in $\mathbb R^n$. Then, for every $\varepsilon>0$ there exist bounded smooth convex domains $\Omega_{\varepsilon}$ such that $\Omega\subset\Omega_{\varepsilon}$ and $d^{\mathcal{H}}(\Omega,\Omega_{\varepsilon})<\varepsilon$.
\end{lemme}
The statement of Lemma \ref{approximation_lem} holds true also if we replace $\Omega\subset\Omega_{\varepsilon}$ by $\Omega_{\varepsilon}\subset\Omega$. We also recall the following lemma, see e.g., \cite{ADR}.
\begin{lemme}
Let $\Omega$ be a bounded convex domain in $\mathbb R^n$ and let $\{\Omega_k\}_{k=1}^{\infty}$ be a sequence of bounded smooth convex domains such that $\lim_{k\to\infty}d^{\mathcal{H}}(\Omega,\Omega_k)=0$. Then
\begin{enumerate}[i)]
\item $\lim_{k\to\infty}\rho(k)\to\rho$;
\item $\lim_{k\to\infty}D(k)=D$.
\end{enumerate}
Here $\rho(k),\rho$ and $D(k),D$ denote respectively the inradius and the diameter of $\Omega_k,\Omega$.
\end{lemme}

We are ready to prove Theorem \ref{mainthm}. 

\begin{proof}[Proof of Theorem \ref{mainthm}]

The proof is divided in three steps.

\medskip
{\bf Step 1.} Thanks to Lemma \ref{approximation_lem} we can consider a sequence of bounded smooth convex domains $\{\Omega_k\}_{k=1}^{\infty}$ such that $\Omega\subset\Omega_{k+1}\subset\Omega_k$ for all $k\in\mathbb{N}$, and such that $\lim_{k\to\infty}d^{\mathcal{H}}(\Omega_k,\Omega)=0$. Let $\sigma_j(k),\sigma_j$ denote  the Steklov eigenvalues of $\Omega_k,\Omega$ respectively. We claim that the following statements hold:
\begin{enumerate}[i)]
\item $\sigma_j(k)\to\sigma_j$ for all $j\in\mathbb N$.
\item For any $k\in\mathbb N$, let $\{u_j(k)\}_{j=1}^{\infty}$ be  a sequence of  Steklov eigenfunctions on $\Omega_k$ with traces forming a orthonormal basis of $L^2(\partial\Omega_k)$. Then there exists a sequence  $\{u_j\}_{j=1}^{\infty}$ of  Steklov eigenfunctions on $\Omega$ with traces forming a orthonormal basis of $L^2(\partial\Omega)$ such that, possibly passing to a subsequence with respect to $k$,
\begin{equation}\label{acca1}
\lim_{k\to\infty}\|u_j(k)-u_j\|_{H^1(\Omega)}=0.
\end{equation} 
\end{enumerate}
Claims i) and ii) follow directly either by \cite[Theorem 3.1 and Proposition 4.4]{ferr_lamb}  or by  \cite[Theorem 4.1]{bucgiatre}.  We note that the application of the stability results in \cite{bucgiatre,ferr_lamb} requires the condition $|\partial \Omega_k|\to |\partial \Omega|$ as $k\to \infty $, which is satisfied since the perimeter is continuous with respect to the Hausdorff distance in the class of bounded convex sets, see e.g., \cite[Theorem 2.4.10]{HP}. Note also that \cite[Lemma~3.2.3.2]{grisvard} allows to represent the domains $\Omega_k, \Omega$ in the same `atlas class' as required in \cite{ferr_lamb}. (See also Appendix  \ref{app:proof} for a self-contained proof  of claims i) and ii).)

{\bf Step 2.} Consider any domain $\omega$ with $\overline\omega \subset \Omega$, briefly  $\omega\subset\subset\Omega$. From the previous step we have $\lim_{k\to\infty}\|u_j(k)-u_j\|_{H^1(\omega)}=0$. Now, we have $\Delta(u_j(k)-u_j)=0$ in $\omega$, and by elliptic regularity (see e.g., \cite[Theorem 8.10]{GT}) we infer that $\lim_{k\to\infty}\|u_j(k)-u_j\|_{H^2(\omega)}=0$. In particular
$$
\int_{\omega}|D^2u_j|^2=\lim_{k\to\infty}\int_{\omega}|D^2u_j(k)|^2\leq \lim_{k\to\infty}2\sigma_j(k)C_{D(k),\rho(k),\sigma(k)}=2\sigma_jC_{D,\rho,\sigma_j}.
$$
Here $D(k),\rho(k)$ are the diameter and inradius of $\Omega_k$, while $D,\rho$ are the diameter and inradius of $\Omega$.

{\bf Step 3.} Consider now a sequence of domains $\omega_k\subset\subset\Omega$ such that $\omega_k\subset \omega_{k+1}$, $\cup_{k=1}^{\infty}\omega_k=\Omega$.  Then it follows by the Monotone Convergence Theorem applied to $\chi_{\omega_k}|D^2u|^2$ that
$$
\int_{\Omega}|D^2u_j|^2\leq 2\sigma_jC_{D,\rho,\sigma_j},
$$
hence $u_j\in H^2(\Omega)$. Since the traces of $\{u_j\}_{j=1}^{\infty}$ form  a complete system  in $L^2(\partial\Omega)$, all  other eigenfunctions are  linear combinations of a finite number of those eigenfunctions  hence they belong to $H^2(\Omega)$.
\end{proof}

\begin{rem}[An alternative use of the spectral stability for the proof of Theorem~\ref{mainthm}]\label{alternative}
In Step 1 of the proof of Theorem \ref{mainthm} two results are used: the convergence of the eigenvalues and the convergence of the eigenfunctions. The convergence of the eigenfunctions in Claim {\rm ii)} of Step 1, can be addressed in an alternative way. In fact, one may start by choosing a sequence $\{u_j\}_{j=1}^{\infty}$ of Steklov eigenfunctions in the reference domain $\Omega$, with traces forming a orthonormal basis of $L^2(\partial\Omega)$, and ask  if there exist sequences $\{u_j(k)\}_{j=1}^{\infty}$ for which \eqref{acca1} holds true. In general, this is not necessarily true for multiple eigenvalues, in which case the definition of {\rm generalized eigenfunctions} is required. Namely, given a finite set of Steklov eigenvalues $\sigma_j(k),...,\sigma_{j+m-1}(k)$ on $\Omega_k$ with $\sigma_j(k)\ne\sigma_{j-1}(k)$ and $\sigma_{j+m-1}(k)\ne\sigma_{j+m}(k)$, we call generalized eigenfunction (associated with $\sigma_j(k),...,\sigma_{j+m-1}(k)$) any linear combination of eigenfunctions associated with $\sigma_j(k),...,\sigma_{j+m-1}(k)$. It follows  from \cite[Theorem 3.1 and Proposition 4.4]{ferr_lamb} that, given $\sigma_j,...,\sigma_{j+m-1}$ with $\sigma_j\ne\sigma_{j-1}$ and $\sigma_{j+m-1}\ne\sigma_{j+m}$, and corresponding $L^2(\partial\Omega)$-orthonormal eigenfunctions $u_j,...,u_{j+m-1}$, for all $k\geq k_0$ (with $k_0$ large enough) there exist $m$ $L^2(\partial\Omega)$-orthonormal generalized eigenfunctions $u_j(k),...,u_{j+m-1}(k)$ such that $\lim_{k\to\infty}\|u_i(k)-u_i\|_{H^1(\Omega)}=0$ for all $i=j,...,j+m-1$. Then, using the same argument in Steps 2, 3 of the proof of Theorem~\ref{mainthm}, we conclude that all eigenfunctions $u_j$ of the chosen sequence belong to $H^2(\Omega)$.  
\end{rem}

\section{Comparison with the Dirichlet, Neumann and Robin problems}\label{app:A}

We briefly discuss here why in the case of the Dirichlet, Neumann and Robin (with positive parameter) problems, the proof of the $H^2$-regularity is easier and does not require the the Rellich-Pohozaev identity. In fact, there is no need to estimate any boundary $L^2$-norms.

\medskip

Let $\Omega$ be a bounded smooth convex domain in $\mathbb R^n$. The following estimates hold. 
\begin{itemize}
\item If $u$ is a smooth solution of $$\begin{cases}-\Delta u=\lambda u\,, & {\rm in\ }\Omega\,,\\u=0\,, & {\rm on\ }\partial\Omega,\end{cases}$$ with $\int_{\Omega}u^2=1$ (Dirichlet problem), then from Reilly's formula \eqref{reilly_for} we immediately get
$$
\int_{\Omega}|D^2u|^2\leq\lambda^2;
$$
\item If $u$ is a smooth solution of $$\begin{cases}-\Delta u=\mu u\,, & {\rm in\ }\Omega\,,\\\partial_{\nu}u=0\,, & {\rm on\ }\partial\Omega,\end{cases}$$ with $\int_{\Omega}u^2=1$ (Neumann problem), then from Reilly's formula \eqref{reilly_for} we immediately get
$$
\int_{\Omega}|D^2u|^2\leq\mu^2;
$$
\item If $u$ is a smooth solution of 
$$\begin{cases}-\Delta u=\eta u\,, & {\rm in\ }\Omega\,,\\\partial_{\nu}u+\beta u=0\,, & {\rm on\ }\partial\Omega,\end{cases}
$$
with $\int_{\Omega}u^2=1$ and $\beta>0$ (Robin problem with positive parameter), then from Reilly's formula \eqref{reilly_for} we  get
$$
\int_{\Omega}|D^2u|^2\leq\eta^2-2\int_{\partial\Omega}\Delta_{\partial\Omega}u\partial_{\nu}u=\eta^2+2\beta\int_{\partial\Omega}\Delta_{\partial\Omega}uu=\eta^2-2\beta\int_{\partial\Omega}|\nabla_{\partial\Omega}u|^2\leq\eta^2,
$$ that is
$$
\int_{\Omega}|D^2u|^2\leq\eta^2.
$$
We see  that the positive sign of $\beta$ is crucial in this last case.
\end{itemize}
In all the three cases above the bound on the $L^2$-norm of the Hessian  depends only on the eigenvalue. The eigenvalues of the three problems are continuous with respect to the Hausdorff distance in the case of a sequence of bounded convex domains monotonically converging to a given convex domain (from the exterior or the interior). The estimate then ``passes to the limit''  (this is done as in the proof of Theorem \ref{mainthm} in Section \ref{approximation}, or as in \cite[\S3]{grisvard}) and this implies the $H^2$-regularity on the limit domain.

\appendix

\section{Spectral stability of the Steklov problem}\label{app:proof}

The spectral stability of the Steklov problem under domain perturbation has been considered in the papers \cite{bucgiatre,ferr_lamb}. The results in those papers  are very general and sharp and can be applied to our problem as discussed in the proof of Theorem~\ref{mainthm} and in Remark~\ref{alternative}. 

However, since the perturbation problem considered in the present paper is rather simple, the spectral stability result used in the proof of Theorem~\ref{mainthm} can be proved with a short and self-contained argument that  we include here for the reader's convenience. The main tool that we need is the following lemma, which follows from \cite[\S2]{HP}
\begin{lemme}\label{unif_appr}
Let $\Omega$ be a bounded convex domain in $\mathbb R^n$ and let $\Omega_k$, $k\in\mathbb{N}$, be a sequence of bounded smooth convex domains, with $\Omega\subset\Omega_{k+1}\subset\Omega_k$ for all $k\in\mathbb{N}$,  such that $\lim_{k\to+\infty}d^{\mathcal{H}}(\Omega,\Omega_k)=0$. Then for any $u_k\in H^1(\mathbb R^n)$ converging weakly in $H^1(\mathbb R^n)$ to $u\in H^1(\mathbb R^n)$ we have
$$
\lim_{k\to\infty}\int_{\partial\Omega_k}u_k^2=\int_{\partial\Omega}u^2.
$$
\end{lemme}
The spectral stability result is stated as follows.

\begin{thm}\label{stability_easy}
Let $\Omega$ be a bounded convex domain in $\mathbb R^n$ and let $\Omega_k$, $k\in\mathbb{N}$, be a sequence of bounded smooth convex domains, with $\Omega\subset\Omega_{k+1}\subset\Omega_k$ for all $k\in\mathbb{N}$,  such that $\lim_{k\to+\infty}d^{\mathcal{H}}(\Omega,\Omega_k)=0$. Let $\sigma_j(k),\, \sigma_j$ denote  the Steklov eigenvalues of $\Omega_k,\, \Omega$ respectively. Then
\begin{equation}\label{limh1}
\lim_{k\to \infty}\sigma_j(k)=\sigma_j,
\end{equation}
for all $j\in\mathbb N$. 
Moreover, 
let $\{u_j(k)\}_{j=1}^{\infty}$ be a sequence of Steklov eigenfunctions on $\Omega_k$ with traces forming an orthonormal basis of $L^2(\partial\Omega_k)$. Then,   there exists a sequence of  Steklov eigenfunctions on $\Omega$, denoted by $\{u_j\}_{j=1}^\infty$ and associated with $\sigma_j$, with traces forming an orthonormal basis of $L^2(\partial\Omega)$, such that
possibly passing to a subsequence with respect to $k$, we have
\begin{equation}\label{limh2}
\lim_{k\to\infty}\|u_j(k)-u_j\|_{H^1(\Omega)}=0.
\end{equation}
\end{thm}
\begin{proof}
The proof is divided in three steps.

\smallskip

{\bf Step 1.} We fix $j\in \mathbb{N}$ and we prove that $\sigma_j(k)\to\sigma_j$ as $k\to\infty$. Let $v_i$, $i=1,...,j$, be a $L^2(\partial\Omega)$ orthonormal family of Steklov eigenfunctions in $\Omega$, with associated eigenvalues $\sigma_i$. 
Let $\tilde v_i$ be some extension of $v_i$ to $\mathbb R^n$, and let $\tilde V_j=\{\sum_{i=1}^ja_i\tilde v_i:\sum_{i=1}^ja_i^2=1, \ a_i\in \mathbb{R} \}$. We have
$$
\sigma_j(k)\leq\max_{\tilde v\in \tilde V_j}\frac{\int_{\Omega_k}|\nabla \tilde v|^2}{\int_{\partial\Omega_k}\tilde v^2}=\max_{\tilde v\in \tilde V_j}\frac{\int_{\Omega}|\nabla \tilde v|^2+\int_{\Omega_k\setminus\Omega}|\nabla \tilde v|^2}{\int_{\partial\Omega}\tilde v^2}\frac{\int_{\partial\Omega}\tilde v^2}{\int_{\partial\Omega_k}\tilde v^2}\leq(1+\varepsilon'(j,k))\left(\sigma_j+\max_{\tilde v\in\tilde V_j}\frac{\int_{\Omega_k\setminus\Omega}|\nabla \tilde v|^2}{\int_{\partial\Omega}\tilde v^2}\right)
$$
where $\varepsilon'(j,k)\to 0$ as $k\to\infty$. Now, for any $\tilde v\in\tilde V_j$, by the absolute continuity of the Lebesgue integral, we have $\int_{\Omega_k\setminus\Omega}|\nabla\tilde v|^2\to 0$ as $k\to \infty$ and since $\tilde V_j$ is finite dimensional, we get altogether
\begin{equation}\label{EE}
\sigma_j(k)\leq(1+\varepsilon'(j,k))(\sigma_j+\varepsilon''(j,k))
\end{equation}
where $\varepsilon''(j,k)\to 0$ as $k\to\infty$ for fixed $j$. 
Now, consider the restrictions of $u_i(k)$ to $\Omega$. Note that by inequality \eqref{EE} and the normalization of the eigenfunctions, the norm $\|u_i(k)\|_{H^1(\Omega_k)}$ is uniformly bounded with respect to $k$ hence Lemma  \ref{unif_appr} is applicable. Since $\int_{\partial\Omega_k}u_i(k)u_j(k)=\delta_{ij}$, from Lemma \ref{unif_appr} we deduce that for $k\geq k_0$, $u_i(k)$ are linearly independent also in $L^2(\partial\Omega)$. Let $V_j(k)=\{\sum_{i=1}^ja_iu_i(k):\sum_{i=1}^ja_i^2=1, \ a_i\in\mathbb R\}$. We have
$$
\sigma_j\leq\max_{v\in V_j(k)}\frac{\int_{\Omega}|\nabla v|^2}{\int_{\partial\Omega}v^2}\leq\max_{v\in V_j(k)}\frac{\int_{\Omega_k}|\nabla v|^2}{\int_{\partial\Omega_k}v^2}\frac{\int_{\partial\Omega_k}v^2}{\int_{\partial\Omega}v^2}\leq\sigma_j(k)(1+\varepsilon(j,k))
$$
with $\varepsilon(j,k)\to 0$ as $k\to\infty$ for fixed $j$ by Lemma \ref{unif_appr}. 
This combined with \eqref{EE} implies that $\sigma_j(k)\to\sigma_j$ for all $j$ as $k\to\infty$.

\smallskip

{\bf Step 2.} By Lemma~\ref{hessian_smooth} and the  uniform bounds for the norms in $H^1(\Omega_k)$ of the eigenfunctions (see Step 1),  we have that $\{u_j(k)\}_{k=1}^{\infty}$ is bounded in $H^2(\Omega)$. Up to extracting a subsequence, we find $u_j\in H^2(\Omega)$ such that $u_j(k)\to u_j$ in $H^1(\Omega)$. We now show that $u_j$ is an eigenfunction with eigenvalue $\sigma_j$. Let $\phi\in H^1(\Omega)$ and let $\Phi$ be an extension to $\Omega_1$ (or to a fixed box containing all $\Omega_k$). Then
\begin{equation}\label{weak_lim}
\int_{\Omega_k}\nabla u_j(k)\cdot\nabla\Phi=\sigma_j(k)\int_{\partial\Omega_k}u_j(k)\Phi.
\end{equation}
We consider the integral in  the right-hand side of \eqref{weak_lim}, and  we write
$$
\int_{\partial\Omega_{k}}u_j(k)\Phi=\int_{\partial\Omega}u_j\Phi+\left(\int_{\partial\Omega_{k}}u_j(k)\Phi-\int_{\partial\Omega}u_j(k)\Phi\right)+\left(\int_{\partial\Omega}u_j(k)\Phi-\int_{\partial\Omega}u_j\Phi\right).
$$
The second term in the right-hand side goes to zero as $k\to\infty$ thanks to Lemma \ref{unif_appr}, while the third term goes to zero from the compactness of the trace operator. For the left-hand side of \eqref{weak_lim}, we have
\begin{equation}\label{last}
\int_{\Omega_k}\nabla u_j(k)\cdot\nabla\Phi=\int_{\Omega}\nabla u_j\cdot\nabla\Phi+\int_{\Omega_k\setminus\Omega}\nabla u_j(k)\cdot\nabla\Phi+\left(\int_{\Omega}\nabla (u_j(k)-u_j)\cdot\nabla\Phi\right).
\end{equation}
The second term in the right-hand side of \eqref{last} goes to zero as $k\to\infty$ because 
$$
\int_{\Omega_k\setminus\Omega}\nabla u_j(k)\cdot\nabla\Phi\leq\|\nabla u_j(k)\|_{L^2(\Omega_k\setminus\Omega)}\|\nabla \Phi\|_{L^2(\Omega_k\setminus\Omega)}\leq (\sigma_j(k))^{1/2}\|\nabla \Phi\|_{L^2(\Omega_k\setminus\Omega)}
$$
and because $\|\nabla \Phi\|_{L^2(\Omega_k\setminus\Omega)}$ goes to zero as before by the absolute continuity of the Lebesgue integral. The third term in \eqref{last} goes to zero since $u_j(k)\to u_j$ in $H^1(\Omega)$.
In conclusion, 
\begin{equation}
    \int_{\Omega}\nabla u_j\cdot \nabla \phi =\sigma_j \int_{\partial\Omega}u_j\phi
\end{equation}
for all $\phi\in H^1(\Omega)$, and 
 \eqref{limh2} holds.

\smallskip

{\bf Step 3.} 
It remains to prove that $u_j\ne 0$ for all $j\in\mathbb{N}$ and that  the traces of $\{u_j\}_{j=1}^{\infty}$ form an orthonormal basis of $L^2(\partial\Omega)$. This follows simply by passing to the limit in the equality 
$
\int_{\Omega_k}\nabla u_i(k)\cdot \nabla u_j(k)=\sigma_j(k)\delta_{ij}
$
as $k\to \infty$
in order to get 
$
\int_{\Omega}\nabla u_i\cdot \nabla u_j=\sigma_j\delta_{ij}.
$
In fact, writing 
$$\int_{\Omega_k}\nabla u_i(k)\cdot \nabla u_j(k)=\int_{ \Omega}\nabla u_i(k)\cdot \nabla u_j(k)+\int_{\Omega_k\setminus \Omega}\nabla u_i(k)\cdot \nabla u_j(k)
$$
we have that the first integral in  the right-hand side of the previous equality converges to 
$
\int_{\partial \Omega}\nabla u_i\cdot \nabla u_j
$
while the second integral converges to zero as $k\to \infty$. Indeed, by the H\"{o}lder's inequality, the Sobolev Embedding Theorem and the uniform bound on the norms of the eigenfunctions in $H^2(\Omega_k)$ we get
\begin{multline}\left|
\int_{\Omega_k\setminus \Omega}\nabla u_i(k)\cdot \nabla u_j(k)\right|\le 
\| \nabla u_i(k)\|_{L^2(\Omega_k\setminus \Omega)}\| \nabla u_j(k)\|_{L^2(\Omega_k\setminus \Omega)}\\
\le |\Omega_k\setminus\Omega|^{1-2/p}\| \nabla u_i(k)\|_{L^p(\Omega_k\setminus \Omega)}\| \nabla u_j(k)\|_{L^p(\Omega_k\setminus \Omega)}\le C |\Omega_k\setminus\Omega|^{1-2/p}
\end{multline}
for some $p>2$, where $C$ is independent of $k$ (note that the constants arising from the Sobolev Embedding Theorem  do not depend on $k$ since \cite[Lemma~3.2.3.2]{grisvard} allows to represent the domains $\Omega_k$ in the same `atlas class').

\end{proof}





\section*{Acknowledgments}
The first author is a member of the Gruppo Nazionale per l'Analisi  Matematica, la Probabilit\`{a} e le loro Applicazioni (GNAMPA) of the Istituto Nazionale di Alta Matematica (INdAM). The second author acknowledges the support of the INdAM GNSAGA. The authors aknowledge financial support from the project ``Perturbation problems and asymptotics for elliptic differential equations: variational and potential theoretic methods'' funded by the European Union – Next Generation EU and by MUR-PRIN-2022SENJZ3. 

\bibliography{bibliography}{}

\begin{thebibliography}{10}

\bibitem{bucgiatre}
D.~Bucur, A.~Giacomini, and P.~Trebeschi.
\newblock {$L^\infty$} bounds of {S}teklov eigenfunctions and spectrum
  stability under domain variation.
\newblock {\em J. Differential Equations}, 269(12):11461--11491, 2020.

\bibitem{GirouSurv}
B.~Colbois, A.~Girouard, C.~Gordon, and D.~Sher.
\newblock Some recent developments on the {S}teklov eigenvalue problem.
\newblock {\em Rev. Mat. Complut.}, 37(1):1--161, 2024.

\bibitem{ADR}
A.~Delyon, A.~Henrot, and Y.~Privat.
\newblock The missing {$(A,D,r)$} diagram.
\newblock {\em Ann. Inst. Fourier (Grenoble)}, 72(5):1941--1992, 2022.

\bibitem{ferr_lamb}
A.~Ferrero and P.~D. Lamberti.
\newblock Spectral stability of the {S}teklov problem.
\newblock {\em Nonlinear Anal.}, 222:Paper No. 112989, 33, 2022.

\bibitem{GT}
D.~Gilbarg and N.~S. Trudinger.
\newblock {\em Elliptic partial differential equations of second order}.
\newblock Classics in Mathematics. Springer-Verlag, Berlin, 2001.
\newblock Reprint of the 1998 edition.

\bibitem{Pol17}
A.~Girouard and I.~Polterovich.
\newblock Spectral geometry of the {S}teklov problem (survey article).
\newblock {\em J. Spectr. Theory}, 7(2):321--359, 2017.

\bibitem{grisvard}
P.~Grisvard.
\newblock {\em Elliptic problems in nonsmooth domains}, volume~24 of {\em
  Monographs and Studies in Mathematics}.
\newblock Pitman (Advanced Publishing Program), Boston, MA, 1985.

\bibitem{HP}
A.~Henrot and M.~Pierre.
\newblock {\em Shape variation and optimization}, volume~28 of {\em EMS Tracts
  in Mathematics}.
\newblock European Mathematical Society (EMS), Z\"{u}rich, 2018.
\newblock A geometrical analysis, English version of the French publication [
  MR2512810] with additions and updates.

\bibitem{pohozaev}
S.~I. Poho\v{z}aev.
\newblock On the eigenfunctions of the equation {$\Delta u+\lambda f(u)=0$}.
\newblock {\em Dokl. Akad. Nauk SSSR}, 165:36--39, 1965.

\bibitem{PS_steklov}
L.~Provenzano and J.~Stubbe.
\newblock Weyl-type bounds for {S}teklov eigenvalues.
\newblock {\em J. Spectr. Theory}, 9(1):349--377, 2019.

\bibitem{Reilly}
R.~C. Reilly.
\newblock Applications of the {H}essian operator in a {R}iemannian manifold.
\newblock {\em Indiana Univ. Math. J.}, 26(3):459--472, 1977.

\bibitem{rellich}
F.~Rellich.
\newblock Darstellung der {E}igenwerte von {$\Delta u+\lambda u=0$} durch ein
  {R}andintegral.
\newblock {\em Math. Z.}, 46:635--636, 1940.

\bibitem{steklov}
W.~Stekloff.
\newblock Sur les probl\`emes fondamentaux de la physique math\'{e}matique
  (suite et fin).
\newblock {\em Ann. Sci. \'{E}cole Norm. Sup. (3)}, 19:455--490, 1902.

\end{thebibliography}
\bibliographystyle{abbrv}
\end{document}